\documentclass{amsart}
\usepackage{amssymb}
\usepackage{stmaryrd} 
\usepackage{amsmath} 
\usepackage{amscd}
\usepackage{amsbsy}
\usepackage{commath, longtable}
\usepackage{comment, enumerate}
\usepackage{geometry}
\usepackage[matrix,arrow]{xy}
\usepackage{hyperref}
\usepackage{mathrsfs}
\usepackage{color}
\usepackage{mathtools,caption}
\usepackage[table,dvipsnames]{xcolor}
\usepackage{tikz-cd}
\usepackage{longtable}
\usepackage[utf8]{inputenc}
\usepackage[OT2,T1]{fontenc}
\usepackage[normalem]{ulem}
\usepackage{hyperref}
\usepackage[customcolors]{hf-tikz}
\hypersetup{
 colorlinks=true,
 linkcolor=DarkOrchid,
 filecolor=blue,
 citecolor=olive,
 urlcolor=orange,
 pdftitle={Selmer-class group Project},
 }
\usepackage{booktabs}

\DeclareSymbolFont{cyrletters}{OT2}{wncyr}{m}{n}
\DeclareMathSymbol{\Sha}{\mathalpha}{cyrletters}{"58}

\DeclareMathOperator{\Norm}{\mathsf{N}}

\DeclareMathOperator{\Gal}{Gal}

\DeclareMathOperator{\Cl}{Cl}

\DeclareMathOperator{\rank}{rank}

\DeclareMathOperator{\Res}{Res}
\DeclareMathOperator{\res}{res}

\DeclareMathOperator{\coker}{coker}
\DeclareMathOperator{\Sel}{Sel}

\DeclareMathOperator{\Image}{Im}
\DeclareMathOperator{\unr}{unr}

\newcommand{\ZZ}{\mathbb Z}
\newcommand{\CC}{\mathbb C}

\newcommand{\cF}{\mathcal F}

\newcommand{\EC}{\mathsf E}

\newcommand{\cO}{\mathcal O}

\newcommand{\fp}{\mathfrak p}
\newcommand{\p}{\mathfrak p}
\newcommand{\fq}{\mathfrak q}

\newcommand{\fQ}{\mathfrak Q}

\newcommand{\Z}{{\mathbb{Z}}}
\newcommand{\Q}{{\mathbb{Q}}}

\newcommand{\F}{{\mathbb{F}}}

\newcommand{\Eab}{\EC_{a,b}}

\newtheorem*{theorem*}{Theorem}
\newtheorem*{conj*}{Conjecture}
\newtheorem*{Ques*}{Question}
\newtheorem*{cor*}{Corollary}
\newtheorem{theorem}{Theorem}[section]
\newtheorem{lemma}[theorem]{Lemma}
\newtheorem{Proposition}[theorem]{Proposition}
\newtheorem{cor}[theorem]{Corollary}

\newtheorem{lthm}{Theorem}

\definecolor{Green}{rgb}{0.0, 0.5, 0.0}

\theoremstyle{definition}
\newtheorem{definition}[theorem]{Definition}
\theoremstyle{remark}
\newtheorem{remark}[theorem]{Remark}
\newtheorem*{Remark*}{Remark}

\makeatletter
\theoremstyle{plain} 
\newtheorem*{intr@thm}{\intr@thmname}

\newtheorem*{c@njecture}{\conjn@name}
\newcommand{\myl@bel}[2]{%
  \protected@write \@auxout {}{\string \newlabel {#1}{{#2}{\thepage}{#2}{#1}{}} }%
  \hypertarget{#1}{}
    } 
    {
        \def\conjn@name{#2}
        \begin{c@njecture}[{#1}]\myl@bel{#3}{#2}
    }
    {
        \end{c@njecture}
    }
\makeatother

\begin{document}
\title[]{Class Groups and Selmer Groups in Special Families}
\author[Abhishek]{Abhishek}
\address[Abhishek]{Department of Mathematics and Statistics\\ IIT Kanpur\\ India, 208016}
\email{abhi.math04@gmail.com}

\author[D.~Kundu]{Debanjana Kundu}
\address[Kundu]{University of Texas Rio Grande Valley \\ Edinburg \\
TX, USA}
\email{debanjana.kundu@utrgv.edu}

\date{\today}


\begin{abstract}
We explore the relationship between (3-isogeny induced) Selmer group of an elliptic curve and the (3 part of) the ideal class group, over certain non-abelian number fields.
\end{abstract}

\maketitle

\section{Introduction}

Often the Selmer group of an elliptic curve and certain ideal class group can be embedded into the same cohomology group; thereby suggesting that in many cases it might be possible to compute one from the other.
This line of thought has been exploited in several papers including \cite{Ca91, Sch96}.

The main aim of this paper is to take this na\"{i}ve idea forward and generalize a result of \cite{jha23, JMS-part2} and explore the relationship between (3-isogeny induced) Selmer group of an elliptic curve and the (3 part of) the ideal class group, over certain dihedral extensions.

Let $\EC$ be an elliptic curve over defined over $\Q$ given by $\EC : y^2 = x^3 + rx^2 + sx + t$, where $r,s,t\in \Q$.
Write $C$ to denote a subgroup of order 3 in $\EC(\overline{\Q})$ which is stable under the action of the absolute Galois group $G_{\Q}$.
By a change of co-ordinates (if necessary), we may assume that $C = \{O, (0,\sqrt{t}),(0,- \sqrt{t})\}$ with $s^2 = 4rt$.
One of the two things can happen:
either $s=0$, in which case $r=0$ and $\EC = \EC_t: y^2 = x^3 + t$, where $t \in \Z\setminus \{0\}$; such an elliptic curve will be said to be of Type I.
These curves were studied extensively by L.~Mordell and are often also called \emph{Mordell curves} in the literature.
Otherwise $s\neq 0$ in which case (by change of variables, if necessary) the elliptic curve $\EC = \EC_{a,b}: y^2 = x^3 + a(x-b)^2$ satisfying $a,b \in \Z$ and $ab(4a+27b)\neq 0$.
Such an elliptic curve will be said to be of Type II.
In either case, one can show that there exists a rational $3$-isogeny $\varphi: \EC \rightarrow \EC/C$, which is described in detail in Sections~\ref{S: type I} and \ref{S: type II}, respectively.

\subsection*{Overview of results}
In \cite{jha23, JMS-part2}, the authors studied the close relationship between the $\F_3$-dimension of the $\varphi$-Selmer group of the elliptic curves $\EC_a$ or $\EC_{a,b}$ over $\Q(\zeta_3)$ and the 3-part of the $S$-class group of $\Q(\zeta_3, \sqrt{a})$ (resp. $\Q(\zeta_3)$) depending on whether $a$ is or is not a square in $\Q(\zeta_3)^*$.
This article extends the above results to dihedral extensions.
In other words, for a fixed cube-free integer $m$ we set $K = \Q(\zeta_3, m^{1/3})$ and we study the explicit upper and lower bounds of the $\F_3$-dimension of the $\varphi$-Selmer group of the elliptic curves $\EC_a$ or $\EC_{a,b}$ over $K$ (essentially) in terms of the 3-part of certain $S$-class group of $K(\sqrt{a})$ (resp. $K$) depending on whether $a$ is or is not a square in $K^*$.
Our main results are presented in Theorems~\ref{theorem relating phi-Selmer group to class group} and \ref{main thm Type II}.

More precisely, for Type I curves our main result is the following:

\begin{lthm}
Let $K$ be as defined before and set $L = K(\sqrt{a})$ (resp. $L = K\times K$) when $a\not\in K^{*2}$ (resp. $a\in K^{*2}$).
Let $S(K)= S_a(K)\cup S_3(K)$ where $S_3(K)$ denotes the set of primes above $3$ in $K$ and $S_a(K) = \{ \fq\in\Sigma_K\mid a\in K_\fq^{\ast 2}\text{ and }v_\fq(4a)\not\equiv 0\pmod{6}\}$ and when $a\not\in K^{*2}$ write $S_a(L)$ to denote the primes in $L$ above $S_a(K)$.
Set $\abs{S}$ to denote the size of the set $S$.
\begin{enumerate}[\textup{(}i\textup{)}]
\item If $a\notin K^{*2}$, then
\[
\dim_{\F_3}\Cl_{S}(K)[3]\leq \dim_{\F_3}\Cl_{S}(L)[3]\leq \dim_{\F_3}\Sel^\varphi(\EC_a/K)\leq \dim_{\F_3}\Cl_{S}(L)[3]+\abs{S(L)}+r_1+r_2.
\]
\item If $a\in K^{*2}$, then 
\[
\dim_{\F_3}\Cl_{S}(K)[3]\leq \dim_{\F_3}\Sel^\varphi(\EC_a/K)\leq \dim_{\F_3}\Cl_{S}(K)[3]+\abs{S(K)}+2.
\]
\end{enumerate} 
\end{lthm}

From an arithmetic perspective, the $3$-Selmer group is often more relevant and interesting.
We can in fact prove a precise relationship between the $S$-class group and the $3$-Selmer group.
This is the content of Corollary~\ref{a finer bound for Type I}.

\begin{cor*}
Let $\Sel^3(\EC_a/K)$ denote the $3$-Selmer group  of $\EC_a/K$.
\begin{enumerate}[\textup{(}i\textup{)}]
\item If $a\notin K^{\ast 2}$, then
    \[
\max\{\dim_{\F_3}\Cl_{S}(L)[3], \operatorname{rk}(\EC_a(K))\}\leq \dim_{\F_3}\Sel^3(\EC_a/K)\leq 2(\dim_{\F_3}\Cl_{S}(L)[3]+\abs{S(L)}+r_1+r_2).
\]
\item If $a\in K^{\ast 2}$, then
\[
 \operatorname{rk}(\EC_a(K))\leq \dim_{\F_3}\Sel^3(\EC_a/K)\leq 2(\dim_{\F_3}\Cl_{S}(K)[3]+\abs{S(K)}+2).
\]
\end{enumerate}
\end{cor*}

For Type II curves, we prove similar results.
The only difference is that the description of the set $S$ involved is more involved.
We refer the reader to the main article for the same.

The main idea of the proof is that the global cohomology group $H^1(G_K, \EC[\varphi])$ can be viewed as a kernel of a certain norm map.
There is a similar description for the local Galois cohomology groups which appear in the definition of the $\varphi$-Selmer group.
To arrive at the bounds, we perform explicit computations of the image of the local Kummer maps and use standard results from algebraic number theory.
In general, these calculations are tedious.
However for our choice of elliptic curves and number fields, the calculations can be done rather explicitly.

Moreover, as an application of our main results described above, we can show (using a result of F.~Gerth) that the $\varphi$-Selmer groups over $K$ become arbitrarily large as $m$ varies.
More precisely, given a Type I curve $\EC_a$ and a fixed integer $n$, using our result we can explicitly find a number field $K$ such that $\dim_{\F_3}\Sel^\varphi(\EC_a/K)\ge n$.
In particular, we prove the following result in Corollary~\ref{cor for Type I curve}.

\begin{cor*}
Keep the notation introduced above.
Write $m=3^s\underset{i=1}{\prod}^Jq_i^{t_i}$, where $t_i=1 \text{ or } 2$ for all $1\leq i\leq J$ and $s=0, 1, \text{ or } 2$ and $q_i\equiv -1\pmod{3}$.
Further define 
\[
r=\begin{cases}
J & \text{if every } q_i\equiv -1\pmod{9} \text{ and } m\not\equiv\pm 1\pmod{9}\\
J-1& \begin{cases}
     \text{if every } q_i\equiv -1\pmod{9} \text{ and } m\equiv\pm 1\pmod{9}\\
     \text{if some } q_i\equiv 2 \text{ or } 5\pmod{9} \text{ and } m\not\equiv\pm 1\pmod{9}
     \end{cases}\\
  J-2 & \text{if some } q_i\equiv 2 \text{ or } 5\pmod{9} \text{ and } m\equiv\pm 1\pmod{9}   
\end{cases}.
\]
Then for all $a$,
\[
2r-3\abs{S(K)}\leq \dim_{\F_3} \Cl_{S}(K)[3] \leq \dim_{\F_3}\Sel^\varphi(\EC_a/K).
\]
In particular, as $m$ varies over all possible values, $\dim_{\F_3}\Sel^\varphi(\EC_a/K)$ becomes arbitrarily large.
\end{cor*}

\subsection*{Organization}
Including this introduction, the article has five sections.
Section~\ref{S: number theoretic prelims} is preliminary in nature -- we explain the basic setup, set notation, make definitions, and prove elementary results which are required throughout the paper.
Section~\ref{S: type I} concerns our main results for Type I elliptic curves.
In Section~\ref{S: type II} we prove similar results but now concerning rational elliptic curves of Type II.
This requires explicit computation of the Tamagawa numbers using Tate's algorithm.
We provide examples wherever possible to elucidate our theorem(s).
In Section~\ref{S: applications} we combine our results with results of Gerth to prove that the 3-part of certain $\varphi$-Selmer groups can be made arbitrarily large over certain number fields.

\subsection*{Acknowledgements}
We thank Somnath Jha and Sudhanshu Shekhar for their helpful comments and continued support.
We thank the referee for their timely report and for pointing out the errors in the first version.
DK acknowledges the support of an AMS--Simons Early Career Travel Grant.

\section{Number Theoretic Preliminaries}
\label{S: number theoretic prelims}
For any number field $\cF$, we write $\Sigma_{\cF}$ to denote the set of all non-Archimedean places of $\cF$.

Fix an integer $m$ and set $K:=\mathbb{Q}(\mu_3, m^{1/3})$.
The number field $K/\Q$ is a dihedral extension with Galois group $\Gal(K/\Q)\simeq D_3$ (but recall that $D_3$ is the symmetric group on 3 elements).
Set $L:=\frac{K[x]}{\langle x^2-a\rangle}$; if $a\notin K^{\ast 2}$ then $L$ is a quadratic extension of $K$, otherwise $L\cong K\times K$.
Next, set $A:=\frac{\cO_K[x]}{\langle x^2-a\rangle}$; if $a\notin K^{\ast 2}$ then $A=\cO_L$, otherwise $A\cong \cO_K\times \cO_K$.
For a prime $\fq$ in $K$, we similarly define $L_{\fq} := \frac{K_{\fq}[x]}{\langle x^2 - a\rangle}$.
If $a\notin K_\fq^{\ast 2}$, then $L_\fq /K_\fq$ is a quadratic extension, otherwise $L_\fq:=K_\fq \times K_\fq$.
For a field $F$, we denote the maximal unramified extension of $F$ by  $F^{\unr}$.

Further, set $A_\fq= \cO_{L_{\fq}}$ when $L_\fq$ is a field.
Else, write $A_\fq= \cO_{K_{\fq}}\times \cO_{K_{\fq}}$.
Write $\Norm_{L/K}: L^\ast\longrightarrow K^\ast$; this is the field norm map when $L$ is a field and the multiplication of co-ordinate when $L\cong K\times K$.
This induces a natural map ${\Norm}_{L_\fq/K_\fq}: \cO^*_{L_\fq} \longrightarrow \cO^*_{K_\fq}$ and also $\overline{\Norm}_{L_\fq/K_\fq}: \cO^*_{L_\fq}/\cO^{*3}_{L_\fq} \longrightarrow \cO^*_{K_\fq}/\cO^{*3}_{K_\fq}$.

The following result calculates precisely the size of 
\[
\left({A_{\mathfrak{q}}^\ast}/{A_\mathfrak{q}^{\ast 3}}\right)_{\Norm =1} = \ker\left(\overline{\Norm}_{L_\fq/K_\fq}: \cO^*_{L_\fq}/\cO^{*3}_{L_\fq} \longrightarrow \cO^*_{K_\fq}/\cO^{*3}_{K_\fq} \right).
\]

\begin{Proposition}\label{size of A/A^3}
Let $\fq\nmid 3$ be a prime in K.
Then,
\[
\abs{\left({A_{\mathfrak{q}}^\ast}/{A_\mathfrak{q}^{\ast 3}}\right)_{\Norm =1}}= \begin{cases}
1 & \text{if } a\notin K_\fq^{\ast 2} \\
3 & \text{if }  a\in K_\fq^{\ast 2}.
\end{cases}
\]
Similarly, if $\mathfrak{p}\mid 3$ is a prime in K, then
\[
\abs{\left({A_{\mathfrak{p}}^\ast}/{A_\mathfrak{p}^{\ast 3}}\right)_{\Norm=1}}= \begin{cases}
3^6 & \text{if } a\notin K_\fp^{\ast 2} \\
3^7 & \text{if }  a\in K_\fp^{\ast 2}.
\end{cases}
\]
\end{Proposition}

\begin{proof}
Suppose that $\fq$ be any prime in $K$ (including $\fq\mid 3$) and that $a\notin K_\fq^{\ast 2}$.

The norm map $\Norm_{L/K}$ induces the map
\[
\Norm_{L_{\fq}/K_{\fq}} \colon \cO^\ast_{L_\fq} \longrightarrow \cO^\ast_{K_\fq}
\]
with $\coker(\Norm_{L_\fq/K_\fq})=\cO^\ast_{K_\fq}/\Norm_{L_\fq/K_\fq}(\cO^\ast_{L_\fq})$.
For an element $z\in \cO^\ast_{K_\fq}$, write $\overline{z}$ to denote the corresponding element in $\coker(\Norm_{L_\fq/K_\fq})$.

Let $x\in\cO_{K_\fq}^\ast \subseteq \cO_{L_\fq}^\ast$, we know that $\Norm_{L_\fq/K_\fq}(x)=x^2$, hence in the cokernel of the norm map we have  $\overline{\Norm_{L_\fq/K_\fq}(x)} = \overline{x^2} = 1$.
This shows that $\overline{x}$ has order $2$.
Hence the cokernel of the norm map $\Norm_{L_\fq/K_\fq}$ is a 2-primary group.
Consider the commutative diagram, 
\[
\begin{tikzcd}
1 \arrow[r, ]  & \cO_{L_\fq}^{\ast 3} \arrow[r, ] \arrow[d, "\Norm_{L_{\fq}/K_{\fq}}"] & \cO_{L_\fq}^\ast \arrow[r, ] \arrow[d, "\Norm_{L_\fq/K_\fq}"] & \cO_{L_\fq}^\ast/\cO_{L_\fq}^{\ast 3}  \arrow[d, "\overline{\Norm}_{L_\fq/K_\fq}"] \arrow[r, ] & 1  \\
1 \arrow[r,] & \cO_{K_\fq}^{\ast 3} \arrow[r, ] & \cO_{K_\fq}^\ast \arrow[r,] & \cO_{K_\fq}^\ast/\cO_{K_\fq}^{\ast 3}  \arrow[r, ] & 1 
\end{tikzcd}
 \]
Since the cokernel of the middle map is a 2-group, it follows from the snake lemma that the map $\overline{\Norm}_{L_\fq/K_\fq}: \frac{\cO_{L_\fq}^\ast}{\cO_{L_\fq}^{\ast 3}}\longrightarrow\frac{\cO^\ast_{K_\fq}}{\cO_{K_\fq}^{\ast 3}}$ is surjective.
Then
\begin{align}
\label{eq: a not in Kq*2}
\abs{\left({A_{\mathfrak{q}}^\ast}/{A_\mathfrak{q}^{\ast 3}}\right)_{\Norm=1}}=\frac{\abs{{\cO_{L_\fq}^\ast}/{\cO_{L_\fq}^{\ast 3}}}}{\abs{{\cO^\ast_{K_\fq}}/{\cO_{K_\fq}^{\ast 3}}}}.
\end{align} 

When $a\in K_\fq^{\ast 2}$, the norm map $\Norm_{L_\fq/K_\fq}:\cO_{K_\fq}^\ast\times \cO_{K_\fq}^\ast\longrightarrow \cO_{K_q}^\ast$ is the multiplication of co-ordinates.
Hence, this map is surjective.
Then, 
\begin{align}
\label{eq: a in Kq*2}
\abs{\left({A_{\mathfrak{q}}^\ast}/{A_\mathfrak{q}^{\ast 3}}\right)_{\Norm=1}} = \frac{\abs{{\cO_{K_\fq}^\ast}/{\cO_{K_\fq}^{\ast3}}\times{\cO_{K_\fq}^\ast}/{\cO_{K_\fq}^{\ast 3}}}}{\abs{{\cO^\ast_{K_\fq}}/{\cO_{K_\fq}^{\ast 3}}}} = \abs{{\cO^\ast_{K_\fq}}/{\cO_{K_\fq}^{\ast 3}}}.
\end{align}

We will write $F$ to denote either $L_\fq$ or $K_\fq$.
When $F$ is a field, we write $\kappa_F$ to denote the residue field of $F$ and let $q$ be the characteristic of this residue field.

\emph{Claim: } Suppose that $q\ne 3$.
Then $\abs{\kappa_F} \equiv 1 \pmod 3$.

\emph{Justification:}
If $q \equiv 1\pmod{3}$, then $\abs{\kappa_F}\equiv 1\pmod{3}$.

Next, suppose that the rational prime $q \equiv 2\pmod{3}$.
Then $q$ is inert in $\Q(\mu_3)$ and (by abuse of notation) we continue to denote this by $q$.
We note that
\[
\abs{\kappa_{{\mathbb{Q}(\mu_3)}_q}}=q^{2} \equiv 1\pmod{3}.
\]
If $q\nmid m$, then $q$ is unramified in $K/\Q$; in particular, $q$ is unramified in $K/\Q(\mu_3)$.
Therefore, $\abs{\kappa_{K_\fq}}=q^2\equiv 1\pmod{3}$.
Otherwise, if $q\mid m $, i.e. $q$ remains unramified in $\Q(\mu_3)/\Q$ but ramifies in $K/\Q(\mu_3)$.
We note that the inertia subgroup at $q$ has order 3.
Therefore, in this case $\kappa_{K_\fq}$ has order $q^{2}\equiv 1\pmod{3}$.

Since $[L_\fq:K_\fq]=1$ or  $2$, we have that $\abs{\kappa_{L_\fq}}=\abs{\kappa_{K_\fq}}$ or $\abs{\kappa_{K_\fq}}^2$.
In any case, $\abs{\kappa_{L_\fq}}\equiv 1\pmod{3}$.
This completes the proof of the claim.

Recall that (\cite[Proposition 5.7]{NSW})
\[
\cO^\ast_F\cong \frac{\ZZ}{(\abs{\kappa_F}-1)\ZZ}\times \frac{\ZZ}{q^s\ZZ}\times \ZZ_q^d,
\]
with $d=[F : \mathbb{Q}_{q}]$ and $s$ such that $F$ contains $q^s$-roots of unity.
To evaluate \eqref{eq: a not in Kq*2} and \eqref{eq: a in Kq*2} in the case that $\fq\nmid 3$, we first compute $\abs{\cO^\ast_F / \cO_F^{\ast 3}}$.
Therefore, 
\begin{align*}
\abs{\cO^\ast_F / \cO_F^{\ast 3}}=\left[\frac{\ZZ}{(\abs{\kappa_F}-1)\ZZ} : 3\frac{\ZZ}{(\abs{\kappa_F}-1)\ZZ}\right]\times \left[\frac{\ZZ}{q^{s}\ZZ} : 3\frac{\ZZ}{q^{s}\ZZ} \right]\times\left[\ZZ_{q}^d :3\ZZ_{q}^d\right] =  3\times 1 \times 1 = 3.
\end{align*}
From these observations, we conclude that $\abs{\cO^\ast_F / \cO_F^{\ast 3}}=3$. 
The claim in the proposition follows immediately in the case that $\fq\nmid 3$.

We now consider the second part of the lemma when $\fq =\fp \mid 3$.
Note that $p=3$ is ramified in $K/\Q$ and  
$\abs{\kappa_{K_{\fp}}}=3$; see for example \cite[Table~1]{schneiders1997estimating}.
Since
\[
\cO^\ast_{K_\fp}\cong \frac{\ZZ}{2\ZZ}\times \frac{\ZZ}{3^s\ZZ}\times \ZZ_3^6,
\] 
it follows immediately that $\abs{\cO^\ast_{K_\fp}/\cO^{\ast 3}_{K_\fp}}=3^7$.
A similar computation shows that $\abs{\cO^\ast_{L_\fp}/\cO^{\ast 3}_{L_\fp}}=3^{13}$.
Substituting these values in \eqref{eq: a not in Kq*2} and \eqref{eq: a in Kq*2} completes the proof of the proposition.
\end{proof}

\begin{remark} 
Let $k$ denote a non-archimedean local field and fix a discrete valuation $v$.
When $\overline{x}$ is an element of $k^\ast / k^{\ast 3}$, define $v(\overline{x}):=v(x)\pmod{3}$ for a lift $x\in k^\ast$.
Similarly, when $(\overline{y}, \overline{z})$ is an element of $k^\ast / k^{\ast 3}\times k^\ast / k^{\ast 3}$, define $v(\overline{y}, \overline{z}):=(v(y)\pmod{3}, v(x)\pmod{3})$ for a lift $(y, z)\in k^\ast \times k^\ast$.
\end{remark}

Henceforth, we write $v_\fq$ for the induced valuation on $L_\fq$, for a prime $\fq \in \Sigma_K$.

\begin{lemma}
\label{lemma: characterization of barx in A/A3 norm 1}
Let $\overline{x}\in (L_\fq^\ast /L_\fq^{\ast 3})$.
Then $\overline{x}\in (A_\fq^\ast/A_\fq ^{\ast 3})_{\Norm=1}$ if and only if $v_\fq(\overline{x})\equiv 0\pmod{3}$.
\end{lemma} 

\begin{proof}
The proof is adapted from \cite[Lemma~2.10]{jha23}.

Consider the case that $a\notin K_\fq^{\ast 2}$.
Let $\overline{x}\in (L_\fq^\ast/L_\fq^{\ast 3})_{\Norm=1}$.
First, suppose that $\overline{x}\in(A_\fq^\ast/A_\fq^{\ast 3})_{\Norm=1}$.
Then $v_\fq(x)=0$, for any lift $x$ of $\overline{x}$.
Conversely, if $v_\fq(\overline{x})\equiv 0\pmod 3$, then $\overline{x}=\overline{x'\fq^{3n}}$ for some $x'\in \cO_{L_\fq}^\ast$ and $n\in \ZZ$.
Hence $\overline{x}=\overline{x'}\in A_\fq^\ast/A_\fq^{\ast 3}$.
Since $\Norm_{L_\fq/K_\fq}(x)\in K_\fq^{\ast 3}$ , we deduce that $\Norm_{L_\fq/K_\fq}(x')\in \cO_{K_\fq}^{\ast 3}$.
We conclude that $\overline{x}\in (A_\fq^\ast/A_\fq ^{\ast 3})_{\Norm=1}$.

A similar proof works for the situation when $a\in K_{\fq}^{\ast 2}$.
\end{proof}

\begin{definition}
\label{defi: the sets M and N}
Let $S(K)$ be a finite set of finite primes of $K$.
When $a \notin K^{\ast 2}$, by definition $L=K(\sqrt{a})$ is a number field.
In this case, define $S(L)$ to denote the primes above $S(K)$ in $L$, i.e.,
\[
S(L):=\{\fQ\in \Sigma_L \mid \fQ \cap K\in S(K) \}.
\]
Set $\cO_{S(K)}$ (resp. $\cO_{S(L)}$) to denote the ring of $S(K)$-integers (resp. $S(L)$-integers) in $K$ (resp. $L$).
Define
\begin{align*}
M(L, S, a)&:=\{\overline{x}\in L^\ast / L^{\ast 3} \mid L(\sqrt[3]{x})/L \text{ is unramified and } x\in L_\fq^{\ast 3} \text{ for all }  \fq \in S(K) \},\\
N(L, S, a)&:=\{ \overline{x}\in L^\ast / L^{\ast 3} \mid \langle x\rangle =I^3 \text{ 
for some fractional ideal }  I \text{ of }  \cO_{S(L)} \}.
\end{align*}
When $a\in K^{\ast 2}$, we analogously define
\begin{align*}
M(K, S, a)&:=\{\overline{x}\in K^\ast / K^{\ast 3} \mid K(\sqrt[3]{x})/K \text{ is unramified and } x\in K_\fq^{\ast 3} \text{ for all }  \fq \in S(K) \},\\
N(K, S, a) & :=\{ \overline{x}\in K^\ast / K^{\ast 3} \mid  \langle x\rangle =I^3 \text{ for some fractional ideal }  I \text{ of }  \cO_{S(K)} \}.
\end{align*}
\end{definition}

For a number field $\cF$, we write $\Cl_{S}(\cF)$ to mean the $S(\cF)$-ideal class group of $\cF$.
Recall that the $S(\cF)$-ideal class group is naturally isomorphic to the Galois group of the maximal abelian unramified extension of $\cF$ with the additional property that all the primes of $S(\cF)$ split completely.

\begin{Proposition}\label{3-ranks of M(S, a) and N(S, a)}
Let $F$ denote either $K$ or $L=K(\sqrt{a})$ where $a\notin K^{\ast 2}$
and set $r_1(F)$ (resp. $r_2(F)$) to denote the number of real embedding (resp. pair of complex embeddings) of $F$ in $\CC$.
Then the following relations of $3$-ranks hold:  
\begin{align*}
\dim_{\F_3 }M(F, S, a) &= \dim_{\F_3}\Cl_{S}(F)[3],\\
\dim_{\F_3} N(F, S, a) &= \dim_{\F_3}\Cl_{S}(F)[3]+\abs{S(F)}+ r_1(F)+r_2(F).
\end{align*}
\end{Proposition}

\begin{proof}
The proof is essentially the same as \cite[Proposition~2.15]{jha23}.
The arguments in \emph{op. cit.} work for the first equality in this setting.
For the second equality, we also refer to \emph{op. cit.}
The map $\rho : N(F, S, a)\rightarrow \Cl_{S(F)}(F)[3]$ defined by $\overline{x}\mapsto [I]$ is surjective with $\ker \rho = \cO_{S(F)}^\ast/\cO_{S(F)}^{\ast 3}$.
Writing $r_1(F)$ and $r_2(F)$ to denote the number of real and complex pair embeddings of $F$ in $\CC$, we know by Dirichlet's $S$-unit theorem that $\dim_{\F_3}\cO_{S(F)}^\ast/\cO_{S(F)}^{\ast 3}=\abs{S(F)}+r_1(F)+r_2(F)$.
Therefore, 
\[
\dim_{\F_3} N(F, S, a) = \dim_{\F_3}\Cl_{S}(F)[3]+\abs{S(F)}+ r_1(F)+r_2(F).
\qedhere
\]
\end{proof}

Next, we provide alternate description of the sets $M(-, S,a)$ and $N(-,S,a)$ defined previously.

\begin{Proposition}
\label{M(s,a) and N(s, a)}
Suppose that $S(K)$ is a finite set of primes containing the primes above 3.
If $a\notin K^\ast 2$, then
\begin{align*}
M(L, S, a) &=\{\overline{x}\in ( L^\ast / L^{\ast 3})_{\Norm=1} \mid \overline{x}\in ( A_\fq^\ast / A_\fq^{\ast 3})_{\Norm=1} \text{ if } \fq\notin S(K), \, \overline{x}=\overline{1} \text{ for all } \fq \in S(K)\},\\ 
N(L, S, a) & =\{\overline{x}\in ( L^\ast / L^{\ast 3})_{\Norm=1} \mid \overline{x}\in ( A_\fq^\ast / A_\fq^{\ast 3})_{\Norm=1} \text{ if } \fq\notin S(K)\}.
\end{align*}
When $a\in K^{\ast 2}$, 
\begin{equation*}
\begin{split}
M(K, S, a) & =\{(\overline{x}, \overline{y})\in(K^\ast / K^{\ast 3}\times K^\ast / K^{\ast 3})_{\Norm=1} \mid (\overline{x}, \overline{y})\in ( A_\fq^\ast / A_\fq^{\ast 3})_{\Norm=1} \text{ for all } \fq\notin S(K) \\
& \hspace{6cm} \text{ and } (\overline{x}, \overline{y})=(\overline{1}, \overline{1}) \text{ for all } \fq \in S(K)\},\\
N(K, S, a) & =\{(\overline{x}, \overline{y})\in(K^\ast / K^{\ast 3}\times K^\ast / K^{\ast 3})_{\Norm=1} \mid (\overline{x}, \overline{y})\in ( A_\fq^\ast / A_\fq^{\ast 3})_{\Norm=1} \text{ for all } \fq\notin S(K) \}.
\end{split}
\end{equation*}
\end{Proposition}

\begin{proof}
Let $\overline{x}\in N(L, S, a)$, then $\langle x\rangle = I^3$ for some fractional ideal of $\cO_{S(L)}$ if and only if  $3 \mid v_\fq(\overline{x})$ for all $\fq \notin S(K)$.
By Lemma~\ref{lemma: characterization of barx in A/A3 norm 1} the two description of $N(L, S, a)$ are equivalent.

Next we show the equivalence of the two descriptions of $M(L, S, a)$.
Suppose that $\overline{x}\in M(L, S, a)$, then $L(\sqrt[3]{x}/L)$ is unramified everywhere.
If $\fq\notin S(K)$, then by \cite[Theorem~6.3(i), Ch.1]{Gr03}, $L(\sqrt[3]{x})/L$ is unramified at $\fQ$ above $\fq$ if and only if $3 \mid v_\fQ(x)$.
Now, Lemma~\ref{lemma: characterization of barx in A/A3 norm 1} asserts that $\overline{x}\in (A_\fq^\ast/A_\fq^{\ast 3})_{\Norm=1}$.
If $\fq \in S(K)$, then $\overline{x}\in L_\fq^{\ast 3}$ if and only if $\overline{x}=\overline{1}\in (L_\fq^\ast/L_\fq^{\ast 3})_{\Norm=1}$.

Now we prove the equivalence of two description of $M(K, S, a)$ and $N(K, S, a)$.
We observe that $(\overline{x}, \overline{y})\in (K^\ast / K^{\ast 3}\times K^\ast / K^{\ast 3})_{\Norm=1}$ if and only if $\overline{y}=\overline{x}^{-1}$.
This yields an isomorphism,  
\begin{align*}
(K^\ast / K^{\ast 3}\times K^\ast / K^{\ast 3})_{\Norm=1} &\xrightarrow{\sim} K^\ast/K^{\ast 3}\\
(\overline{x}, \overline{x}^{-1})&\mapsto \overline{x}.
\end{align*}
Using this identification, we get the following isomorphisms, 
\begin{equation*}
\begin{split}
M(K, S, a)&\cong\{(\overline{x}, \overline{x}^{-1})\in (K^\ast / K^{\ast 3}\times K^\ast / K^{\ast 3})_{\Norm=1} \mid K(\sqrt[3]{x})/K \text{ is unramified and }\\
& \hspace{6.5cm} x\in K_\fq^{\ast 3} \text{ for all }  \fq \in S(K) \},\\
N(K, S, a) &\cong\{ (\overline{x}, \overline{x}^{-1})\in (K^\ast / K^{\ast 3}\times K^\ast / K^{\ast 3})_{\Norm=1} \mid  \langle x\rangle =I^3; \ I \text{ a fractional ideal of }  \cO_S(K)\}.
\end{split}
\end{equation*}

Let $(\overline{x}, \overline{x}^{-1})\in M(K, S, a)$, then $K(\sqrt[3]{x})/K$ is unramified everywhere.
The proof that works for $L$ goes through verbatim for $K$.
Hence $3\mid v_\fq(x)\equiv 0\pmod{3}$; equivalently $\overline{x}\in \cO_{K_\fq}^\ast/\cO_{K_\fq}^{\ast 3}$.
Thus $(\overline{x}, \overline{x}^{-1})\in (\cO_{K_\fq}^\ast/\cO_{K_\fq}^{\ast 3}\times \cO_{K_\fq}^\ast/\cO_{K_\fq}^{\ast 3})_{\Norm=1} \cong (A_{\fq}^\ast/A_{\fq}^{\ast 3})_{\Norm=1}$ for all $\fq \notin S(K)$.
If $\fq\in S(K)$, then $x\in K_\fq^{\ast 3}$ and $\overline{x}=\overline{1}$.
This shows that the two description of $M(K, S, a)$ are equivalent.

Let $(\overline{x}, \overline{x}^{-1})\in N(K, S, a)$, then $\langle x \rangle = I^3$ for some fractional ideal in $\cO_{S(K)}$.
A similar argument will show that $3\mid v_\fq(\overline{x})$ for all $\fq\notin S(K)$.
Hence $(\overline{x}, \overline{x}^{-1})\in (A_{\fq}^\ast/A_{\fq}^{\ast 3})_{\Norm=1}$ for all $\fq \notin S(K)$.
This shows the equivalence of two description of $N(K, S, a)$.
\end{proof}

\section{Elliptic Curve of Type I}
\label{S: type I}

\subsection{Some Preliminaries}

We first record a result which is true for any elliptic curve $\EC$ defined over a number field containing $\mu_3$.

\begin{Proposition}
\label{BES prop}
Let $\EC$, $\widehat{\EC}$ be two elliptic curves over $K$ such that $\phi :\EC \longrightarrow \widehat{\EC} $ is a $3$-isogeny.
Then there is an isomorphism $\EC[\phi]\cong \ker(\Res_K^L \mu_3\rightarrow \mu_3)$ of group scheme, and an induced isomorphism 
\[
H^1(G_K, \EC[\phi])\cong (L^\ast/L^{\ast 3})_{\Norm=1}.
\]
The same statement is also true for the pair $(\widehat{\EC}, \widehat{\phi})$.
\end{Proposition}

The above proposition holds for any field with characteristic $0$.
In particular, it holds for $ K_\fq^{\unr}$. 
\begin{proof}
This result follows from \cite[Proposition~2.4]{bhargava2020average}.
For a discussion on this specific formulation of the result, we refer the reader to \cite[Proposition~1.1]{jha23}.
\end{proof}

Consider the rational elliptic curve of Type I,
\[
\EC_a : y^2=x^3 + a, \, a\ne 0
\]
and set $G_{\Q}:=\Gal(\overline{\Q}/\Q)$.
Let $C$ be the $G_{\Q}$-stable subgroup of $\EC_a(\overline{\Q})$ of order 3.
We have a rational 3-isogeny from $\EC_a$ to $\widehat{\EC_a} :=\EC_a/C$ which is an elliptic curve given by $\EC_{-27a}$; see \cite{Velu}.
Hence, we have a rational 3-isogeny 
\begin{align*}
\phi:  \EC_a &\longrightarrow \EC_{-27a}\\
(x,y) & \mapsto \left(   \frac{x^3+4a}{x^2}, \frac{y(x^3-8a)}{x^3} \right).
\end{align*}
Recall that $\theta_c: \EC_a \to \EC_{c^6a}$ is an isomorphism over any number field $\cF$ where $c\in \cF^*$.
Henceforth, we assume that $a$ is sixth power free in $K$.
Note that $\EC_a \cong \widehat{\EC_a}$; see \cite{jha23}.
Set $\mathfrak{p}=1-\zeta_3$; (using the aforementioned isomorphism) we have the following 3-isogeny over $K$
\begin{align*}
\phi:\EC_a &\longrightarrow \EC_a\\
(x,y) & \mapsto \left(   \frac{x^3+4a}{\mathfrak{p}^2x^2}, \frac{y(x^3-8a)}{\mathfrak{p}^3x^3} \right).
\end{align*}
The exact sequence $0\rightarrow\EC_a[\phi]\rightarrow \EC_a \rightarrow \EC_a\rightarrow 0$ induces the Kummer map, 
\[
\delta_{\phi, K}: \EC_a(K)\longrightarrow \EC_a(K)/\phi (\EC_a(K))\hookrightarrow H^1(G_K, \EC_a[\phi])\cong (L^\ast /L^{\ast 3})_{\Norm=1}.
\]
When $\fq$ is a prime of $K$, we can similarly define the local Kummer map $\delta_{\phi, K_\fq}$.

\begin{definition}
The $\phi$-Selmer group of $\EC_a$ over $K$ is denoted by $\Sel^\phi(\EC_a/K)$ and defined as below
\[
\Sel^\phi(\EC_a/K):=\{ \alpha\in H^1(G_K, \EC_a[\phi]) \mid \res_\fq(\alpha)\in \Image(\delta_{\phi, K_\fq}) \text{ for every}\,  \fq \in \Sigma_K \}.
\]
\end{definition}

In view of Proposition~\ref{BES prop}, we may view the $\phi$-Selmer group to be a set of elements in $(L^\ast/L^{\ast 3})_{\Norm=1}$.
Fix an embedding $\iota:K \hookrightarrow K_\fq$ and similarly for $L$.
Note that if $\overline{x}\in (L^\ast/L^{\ast 3})_{\Norm=1}$, then in fact $\overline{x}\in (L_\fq^\ast/L_\fq^{\ast 3})_{\Norm=1}$.
This allows us to provide an alternate description of $\Sel^\phi(\EC_a/K)$, namely 
\[
\Sel^\phi(\EC_a/K)=\{ \overline{x}\in (L^\ast/L^{\ast 3})_{\Norm=1} \mid \overline{x}\in \Image(\delta_{\phi, K_\fq}) \text{ for every}\,  \fq \in \Sigma_K \}.
\]

\subsection{Description of image of Kummer map}
The following theorem gives an explicit description of the image of the Kummer map $\delta_{\phi, K_\fq}$ and size of this image.

\begin{theorem}
\label{structure of kummer image Type 1}
Let $\EC_a$ be an elliptic curve over $K$ and $\phi$ be the $3$-isogeny mentioned as above.
Let $\fq$ be a prime in $K$ and let $P=(x, y)\in \EC_a(K_\fq)$.
\begin{enumerate}[\textup{(}i\textup{)}]
\item If $a\notin K_\fq^{\ast 2}$ then, 
\[
\delta_{\phi, K_\fq}(P)=\begin{cases}
1 & \text{if } P=O \\
y-\sqrt{a} & \text{otherwise}.
\end{cases}
\]
\item If $a\in K_\fq^{\ast 2}$ then, 
\[
\delta_{\phi, K_\fq}(P)=\begin{cases}
1 & \text{if } P=O \\
\left(\frac{1}{2\sqrt{a}}, 2\sqrt{a}\right) & \text{if } P=(0, \sqrt{a}) \\
\left(-\frac{1}{2\sqrt{a}}, -2\sqrt{a}\right) & \text{if } P=(0, -\sqrt{a}) \\
\left(y-\sqrt{a}, y+\sqrt{a}\right) & \text{otherwise}.
\end{cases}
       \]
   \end{enumerate}
   
\end{theorem}

\begin{proof}
These constructions can be found in \cite[Chapters 13 and 14]{Ca91}.
\end{proof}

The following result is true in general and will be applied for elliptic curves of Type I and II.

\begin{lemma}
\label{size of kummer image}
Let $F$ be either $K_\fq$ or $L_\fq$ (when it is the quadratic extension of $K_\fq$) and $\phi$ be an isogeny from $\EC$ to $\widehat{\EC}$.
Then
\[
\abs{ \widehat{\EC}(F)/\phi(\EC(F))}= \frac{\lVert\phi^\prime(0)\rVert^{-1} \times \abs{\EC(F)[\phi]} \times c_{\fq}(\widehat{\EC})}{c_{\fq}(\EC)}.
\]
\end{lemma}

The proof can be found in \cite[Lemma~3.8]{Sch96}.
We clarify some of the notation introduced in the above proposition - we write $c_\fq(\EC)$ to denote the local Tamagawa number at a prime $\fq$, $\phi^\prime(0)$ is the leading coefficient of the power series expansion of $\phi$ on the formal group of $\widehat{\EC}$, and $\lVert\centerdot\rVert$ is the norm on the field, i.e., $\lVert\centerdot\rVert=\fq^{-v_\fq(\centerdot)}$.
Note that $\lVert\phi^\prime(0)\rVert=1$ if $\fq$ does not divide the degree of isogeny $\phi$, see \cite[Lemma~3.8]{Sch96}).
We remind the reader that the left side of the above expression is equal to the size of the image of the Kummer map.

The following lemma holds for any elliptic curve with $3$-isogeny.
\begin{lemma}
\label{structre of kummer im at good prime}
Let $\EC$ be an elliptic curve over $K$ and $\phi:\EC\rightarrow \widehat{\EC}$ be an isogeny of degree 3 defined over $K$.
Let $\fq\nmid 3$ be a prime in $K$ at which $\EC$ has good reduction.
Then 
\[
\delta_{\phi, K}(\widehat{\EC}(K_\fq))\subset (A_\fq^\ast/A_\fq^{\ast 3})_{\Norm=1}.
\]
\end{lemma}

\begin{proof}
We repeat the proof of \cite[Lemma~3.1]{jha23} for the convenience of the reader.

Since $\EC$ has good reduction at $\fq$,
\[
\Image(\delta_{\phi, K_\fq})\subset \ker(H^1(K_\fq, \EC[\phi])\xrightarrow{\res} H^1(K_\fq^{\unr}, \EC[\phi])).
\]
Proposition~\ref{BES prop} asserts that
\[
H^1(K_\fq, \EC[\phi])\cong((L_\fq^{\unr})^\ast/(L_\fq^{\unr})^{\ast 3})_{\Norm=1}.
\]
Using these identification, we have
\[
\Image(\delta_{\phi, K_\fq})\subset \ker(L_\fq^\ast/L_\fq^{\ast 3}\xrightarrow{\inf} (L_\fq^{\unr})^\ast/(L_\fq^{\unr})^{\ast 3}).
\]
We observe that, if $\overline{x}\in \ker(\inf)$ then $x\in (L_\fq^{\unr})^{\ast 3}$.
Hence $L_\fq(\sqrt[3]{x})/L_\fq$ is unramified which is equivalent to saying $v_\fq(x)\equiv 0\pmod 3$, see \cite[Theorem~6.1(i)]{Gr03}, and conversely.

From the above observation we have,  $\delta_{\phi, K}(\widehat{\EC}(K_\fq))\subset (A_\fq^\ast/A_\fq^{\ast 3})_{\Norm=1}.$ 
\end{proof}

\begin{lemma}\label{size of kummer map}
Assume that $\fq\in \Sigma_K$ is a prime such that $\fq\nmid 3$.
Then
\[
\abs{\EC_a(K_\fq)/\phi(\EC_a(K_\fq)} = \abs{\EC_a(K_\fq)[\phi]}=\begin{cases}
        1 & \text{if }a\notin K_\fq^{\ast 2} \\
        3 & \text{if }a\in K_\fq^{\ast 2}.
\end{cases}
\]
\end{lemma}

\begin{proof}
Since $\fq$ does not divide the degree of the isogeny $\phi$, we know that $\Vert\phi^\prime(0)\rVert=1$.
We already know 
that $\EC_a[\phi]=\{O, (0, \sqrt{a}), (0, -\sqrt{a})\}$.
The result follows from above two observation.
\end{proof}

\subsection{Main Results for Type I Elliptic Curves}
When $a\notin K^{\ast 2}$, let $S_a(K)$ be a finite set of primes in $K$ defined as follows 
\begin{equation}
\label{eqn: SaK}
S_a(K)=\{ \fq\in\Sigma_K\mid a\in K_\fq^{\ast 2}\text{ and }v_\fq(4a)\not\equiv 0\pmod{6}\}.
\end{equation}
Further set $S_a(L)$ to denote the primes in $L$ above $S_a(K)$, i.e.,
\[
S_a(L)=\{\fQ\in \Sigma_L\mid \fQ\cap\cO_K\in S_a(K)\}.
\]
If $a\in K^{\ast 2}$, we note that $S_a(K)=\{ \fq\in\Sigma_K\mid v_\fq(4a)\not\equiv 0\pmod{6}\}$.

The following theorem shows a relationship between the $\phi$-Selmer group and the sets $M(-,S,a)$ and $N(-,S,a)$ defined earlier.
This is a crucial step towards our main theorem of this section, where we estimate the $\F_3$-dimension of the $\phi$-Selmer group.

\begin{theorem}\label{ bounds for selmer}
Set $S(K) = S_a(K)\cup S_3(K)$ where $S_3(K)$ denotes the set of primes above $3$ in $K$.
\begin{enumerate}[\textup{(}i\textup{)}]
\item If $a\notin K^{\ast 2}$, then
\[
M(L, S, a)\subset \Sel^\phi(\EC_a/K)\subset N(L, S, a).
\]
\item If $a\in K^{\ast 2}$, then
\[
M(K, S, a)\subset \Sel^{\phi}(\EC_a/K)\subset N(K, S, a).
\]
\end{enumerate}
\end{theorem}

\begin{proof} 
\begin{enumerate}[\textup{(}i\textup{)}]
\item Let $a\notin K^{\ast 2}$.
First we show that $\Sel^\phi(\EC_a/K)\subset N(L, S, a)$.

In view of  Proposition~\ref{M(s,a) and N(s, a)} it suffices to prove that $\delta_{\phi, K_\fq}(\EC_a(K_\fq))\subset (A_\fq^\ast /A_\fq^{\ast 3})_{\Norm=1}$ for all $\fq\notin S(K)$.
Definition of the set $S_a(K)$ in \eqref{eqn: SaK} asserts that for $\fq\notin S(K)$, either of the following statements hold
\begin{enumerate}
\item $a\notin K_\fq^{\ast 2}$ \emph{or}
\item $a\in K_\fq^{\ast 2}$ and $6\mid v_\fq(4a)$.
\end{enumerate}

\begin{itemize}
\item If $a\notin K_\fq^{\ast 2}$, then $\abs{\EC_a(K_\fq)[\phi]}=1$.
Noting that $\lVert\phi^\prime(0)\rVert=1$, it follows from Lemma~\ref{size of kummer map} that $\abs{\delta_{\phi, K_\fq}(\EC_a(K_\fq))}=1$.
Proposition~\ref{size of A/A^3} asserts that that 
\[
\delta_{\phi, K_\fq}(\EC_a(K_\fq)) = (A_\fq^\ast/A_\fq^{\ast 3})_{\Norm=1}= \{1\}.
\]

\item Next suppose that $a\in K_\fq^{\ast 2}$ and $6\mid v_\fq(4a)$, then $L_\fq=K_\fq\times K_\fq$.
Since $a\in K_\fq^{\ast 2}$, we note that $\EC_a(K_\fq)[\phi]=\{ O, (0, \sqrt{a}), (0, -\sqrt{a})\}$.

Let $P\in \EC_a(K_\fq)$.
To show $\delta_{\phi, K_\fq}(\EC_a(K_\fq))\subset (A_\fq^\ast /A_\fq^{\ast 3})_{\Norm=1}$, it is enough to show that $3 \mid v_\fq(\delta_{\phi, K_\fq}(P))$; this follows from Lemma~\ref{lemma: characterization of barx in A/A3 norm 1}.
If $P=(0, \pm \sqrt{a})$, then Theorem~\ref{structure of kummer image Type 1} asserts that $\delta_{\phi, K_\fq}(P)=\left(\pm \frac{1}{2\sqrt{a}}, \pm 2\sqrt{a}\right)$.
The condition $6\mid v_\fq(4a)$ clearly implies that $3\mid v_\fq(2\sqrt{a})$.
Hence, 
\[
v_\fq(\delta_{\phi, K_\fq}(P))=v_\fq\left(\pm \frac{1}{2\sqrt{a}}, \pm 2\sqrt{a}\right)\equiv 0\pmod{3}.
\]
Set $P=(x, y)\in \EC_a(K_\fq)\setminus\EC_a(K_\fq)[\phi]$.
By Theorem~\ref{structure of kummer image Type 1}, $\delta_{\phi, K_\fq}(P)=(y-\sqrt{a}, y+\sqrt{a})$.
Taking valuation we get that
\[
v_\fq(\delta_{\phi, K_\fq}(P))=(v_\fq(y-\sqrt{a}), v_\fq(y+\sqrt{a})).
\]
The calculation can be divided into two further cases:
 \begin{itemize}
\item If $v_\fq(y)\ne v_\fq(\sqrt{a})$, using the relation $x^3=(y-\sqrt{a})(y+\sqrt{a})$ and taking valuation we get $3v_\fq(x)=v_\fq(y-\sqrt{a})+v_\fq(y+\sqrt{a})$.
Since $v_\fq(y)\ne v_\fq(\sqrt{a})$, it follows that
\[
v_\fq(y\pm \sqrt{a})=\min\{ v_\fq(y), v_\fq(\sqrt{a})\}.
\]
Hence, 
\[
3v_\fq(x)= 2\min\{v_\fq(y), v_\fq(\sqrt{a})\} = 2 v_\fq(y\pm \sqrt{a})
\]
and $v_\fq(\delta_{\phi, K_\fq}(P))\equiv 0\pmod{3}$.
\item Note that $\fq=2$ is a prime of good reduction for $\EC_a$.
If $\fq\ne 2$ and $v_\fq(y)=v_\fq(\sqrt{a})$, then $6\mid v_\fq(a)$.
But, $a$ is $6^{\text{th}}$-power free, so $v_\fq(a)=0$.
Note that $y,\sqrt{a} \in \cO_{K_\fq}^\ast$, and at least one of  $v_\fq(y- \sqrt{a})$ and $v_\fq(y+\sqrt{a})$ is equal to 0.
Using the relation $3v_\fq(x)=v_\fq(y- \sqrt{a})+v_\fq(y+ \sqrt{a})$, we conclude that $3\mid v_\fq(y\pm\sqrt{a})$.
The conclusion follows from Lemma~\ref{structre of kummer im at good prime}.
\end{itemize}
\end{itemize}

Next we show that $M(L, S, a)\subset \Sel^\phi(\EC_a/K)$.
As before, it is suffices to show that $(A_\fq^\ast/A_\fq^{\ast 3})_{\Norm=1}\subset \delta_{\phi, K_\fq}(\EC_a(K_\fq))$ for all $\fq\notin S(K)$.
Once again when $\fq\notin S(K)$, we can divide the calculation into the following two cases:
\begin{enumerate}
\item if $a\notin K_\fq^{\ast 2}$, then as discussed above 
\[
\delta_{\phi, K_\fq}(\EC_a(K_\fq)) = (A_\fq^\ast/A_\fq^{\ast 3})_{\Norm=1}= \{1\}.
\]
\item if $a\in K_\fq^{\ast 2}$, it follows from the above discussion that $\delta_{\phi, K_\fq}(\EC_a(K_\fq)) \subset (A_\fq^\ast/A_\fq^{\ast 3})_{\Norm=1}$.
By Lemmas~\ref{size of A/A^3} and \ref{size of kummer map}, we have
\[
\abs{\delta_{\phi, K_\fq}(\EC_a(K_\fq))} = \abs{(A_\fq^\ast/A_\fq^{\ast 3})_{\Norm=1}}=3.
\]
We conclude that $(A_\fq^\ast/A_\fq^{\ast 3})_{\Norm=1}\subset \delta_{\phi, K_\fq}(\EC_a(K_\fq))$.
\end{enumerate}

\item Consider the case that $a\in K^{\ast 2}$.
We will show that $\Sel^\phi(\EC_a/K)\subset N(K, S, a)$.
By Proposition~\ref{M(s,a) and N(s, a)}, it suffices  to show that  $\delta_{\phi, K_\fq}(\EC_a(K_\fq)) \subset (A_\fq^\ast/A_\fq^{\ast 3})_{\Norm=1}$ for all $\fq\notin S(K)$.
Since $\fq\notin S(K)$ (i.e., $6\mid v_\fq(4a)$ and $a\in K_\fq^{\ast 2}$, for every prime $\fq$ of $K$) 
we can repeat the argument of (a) of part (i) to obtain, 
\[
 \delta_{\phi, K_\fq}(\EC_a(K_\fq)) \subset (A_\fq^\ast/A_\fq^{\ast 3})_{\Norm=1}.
\]
To prove that $M(K, S, a)\subset \Sel^\phi(\EC_a/K)$, it suffices to show that $(A_\fq^\ast/A_\fq^{\ast3})_{\Norm=1}\subset\delta_{\phi, K_\fq}(\EC_a(K_\fq))$ for all $\fq\notin S(K)$.
Again by Lemmas~\ref{size of kummer map} and \ref{size of A/A^3}, we have
\[
\abs{\delta_{\phi, K_\fq}(\EC_a(K_\fq))} = \abs{(A_\fq^\ast/A_\fq^{\ast 3})_{\Norm=1}}=3.
\]
This shows that $(A_\fq^\ast/A_\fq^{\ast3})_{\Norm=1}\subset\delta_{\phi, K_\fq}(\EC_a(K_\fq))$.
 This completes the proof of the theorem.
\end{enumerate}
\end{proof}

In the above proof we always assumed that $\fq\notin S(K)$.
In particular, our calculations were performed for $\fq \nmid 3$.
However, it is easy to notice that for the `upper bound' the same proof goes through even when $\fq \mid 3$.
In other words, even when $\fq \mid 3$, we have 
\[
\delta_{\phi, K_\fp}(\EC_a(K_\fp))\subset (A_\fp^\ast/A_\fp^{\ast 3})_{\Norm=1}.
\]
This provides a finer containment result for the Selmer group which we can now record.
The proof of the theorem is exactly same as \cite[Lemma~3.8]{jha23}.

\begin{theorem}
\label{ finer upper bound}
Choose $S_a(K)$ to be the set of primes defined earlier.
\begin{enumerate}[\textup{(}i\textup{)}]
\item If $a\notin K^{\ast 2}$, then
\[
 \Sel^\phi(\EC_a/K)\subset N(L, S_a, a).
\]
\item If $a\in K^{\ast 2}$, then
\[
 \Sel^{\phi}(\EC_a/K)\subset N(K, S_a, a).
\]
\end{enumerate}
\end{theorem}

The following theorem allows us to relate the $\phi$-Selmer group with the 3-part of the class group of the associated number fields.

\begin{theorem}
\label{theorem relating phi-Selmer group to class group}
Let $S(K)=S_a(K)\cup S_3(K)$ and $S(L)$ denote the primes above $S(K)$ in $L$.
\begin{enumerate}[\textup{(}i\textup{)}]
\item If $a\notin K^{*2}$, then
\[
\dim_{\F_3}\Cl_{S}(K)[3]\leq \dim_{\F_3}\Cl_{S}(L)[3]\leq \dim_{\F_3}\Sel^\phi(\EC_a/K)\leq \dim_{\F_3}\Cl_{S}(L)[3]+\abs{S(L)}+r_1+r_2.
\]
\item If $a\in K^{*2}$, then 
\[
\dim_{\F_3}\Cl_{S}(K)[3]\leq \dim_{\F_3}\Sel^\phi(\EC_a/K)\leq \dim_{\F_3}\Cl_{S}(K)[3]+\abs{S(K)}+2.
\]
\end{enumerate} 
\end{theorem}

\begin{proof}
These results follow from Proposition~\ref{3-ranks of M(S, a) and N(S, a)} and Theorem~\ref{ bounds for selmer}.
\end{proof}

\begin{remark}
It is known that these families of elliptic curves have provably large $\phi$-Selmer groups over $\Q$ for `Tamagawa ratio' reasons.
Our observation is somewhat different (e.g. the Tamagawa ratio is 1 for these Type 1 elliptic curves, since we are working over a field containing $\Q(\mu_3)$).   
\end{remark}
In the following result, we prove what it means it terms of the $3$-Selmer group.

\begin{cor}
\label{a finer bound for Type I}
Let $\Sel^3(\EC_a/K)$ denote the $3$-Selmer group  of $\EC_a/K$.
\begin{enumerate}[\textup{(}i\textup{)}]
\item If $a\notin K^{\ast 2}$, then
    \[
\max\{\dim_{\F_3}\Cl_{S}(L)[3], \operatorname{rk}(\EC_a(K))\}\leq \dim_{\F_3}\Sel^3(\EC_a/K)\leq 2(\dim_{\F_3}\Cl_{S}(L)[3]+\abs{S(L)}+r_1+r_2).
\]
\item If $a\in K^{\ast 2}$, then
\[
 \operatorname{rk}(\EC_a(K))\leq \dim_{\F_3}\Sel^3(\EC_a/K)\leq 2(\dim_{\F_3}\Cl_{S}(K)[3]+\abs{S(K)}+2).
\]
\end{enumerate}
\end{cor}

\begin{proof}
The proof is analogous to \cite[Corollary~3.23]{jha23}).

Recall from the construction of $\phi$ in the Introduction, that $\ker\phi = C = \{O, (0, \pm{\sqrt{a}})\}$.
Therefore, we observe that $\EC_a(K)[\phi]={0}$ for $a\notin K^{\ast 2}$.
It now follows form the exact sequence of \cite[Lemma~6.1]{SS03} that
\[
\dim_{\F_3}\Sel^\phi(\EC_a/K)\le\dim_{\F_3}\Sel^3(\EC_a/K)\le 2\dim_{\F_3}\Sel^\phi(\EC_a/K).
\]
The assertion follows from the fact that  $\operatorname{rk}(\EC_a(K))\le \dim_{\F_3} \Sel^3(\EC_a/K)$ and Theorem~\ref{theorem relating phi-Selmer group to class group}.
For the case $(ii)$, we only get that $\dim_{\F_3}\Sel^3(\EC_a/K)\le 2\dim_{\F_3}\Sel^\phi(\EC_a/K)$, hence the result follows.
\end{proof}

\subsection{Example}
We provide an explicit example where $a\in K^{* 2}$.
Take $K=\Q(\mu_3, 30^{1/3})$ and consider the elliptic curve \href{https://www.lmfdb.org/EllipticCurve/Q/27/a/4}{27a3} in the Cremona table, i.e.  
\[
\EC: y^2 = x^3 + 16.
\]
Recall that, $S_{16}(K)=\{\fq\in \Sigma_K \ | \ v_{\fq}(4 \cdot 16)\neq 0\pmod 6\}$.
Here, the ideal $(2)$ in $K$ factorizes as 
\[
(2)= \fq^3.
\]
This gives us that $v_{\fq}(4 \cdot 16)=12$; this implies $\fq\notin S_{16}(K)$.
Therefore $\abs{S_{16}(K)}=0$.
Note that $(3)= \fp^6$ for a principal prime ideal $\fp$.
Now set 
\[
S = S(K)=S_{16}(K)\cup \{\fp\} = \{\fp\}.
\]
Using SAGE we check that $\Cl(K) = \Cl_S(K) \simeq \Z/3\Z \times \Z/3\Z$.
In particular, $\dim_{\F_3}\Cl_{S}(K)[3] = 2$.

Now using Theorem~\ref{theorem relating phi-Selmer group to class group}(ii), we conclude the following, 
\[
2\le\dim_{\F_3}\Sel^\phi(\EC/K)\le 5.
\]
In view of Corollary~\ref{a finer bound for Type I}(ii) we find that
\[
6 =  \operatorname{rk}(\EC_a(K))\leq \dim_{\F_3}\Sel^3(\EC_a/K)\leq 10.
\]

\section{Elliptic curves of type II}
\label{S: type II}
We now shift our focus to the other family of elliptic curves which support a rational 3-isogeny.
Recall that $K=\Q(\mu_3, m^{1/3})$; for this section we assume that $m$ is an odd integer.
As explained in the introduction, these rational elliptic curves have the form
\[
\Eab: y^2=x^3+a(x-b)^2; \qquad ab(4a+27b)\ne 0.
\]
Let $C$ be a subgroup of $\Eab(\overline{\Q})$ of order 3.
There is a rational $3$-isogeny 
\[
\psi : \Eab\longrightarrow \widehat{\Eab}:=\Eab/C.
\]
By a change of variable $(x, y)\mapsto (\frac{x}{9}-\frac{4a}{27}, \frac{y}{27})$ we get, $\widehat{\Eab}= \EC_{-27a, 4a+27b}$.
The aforementioned isogeny $\psi$ is given by 
\begin{equation}
\label{psi-isogeny}
\begin{split}
\psi:\Eab&\longrightarrow \widehat{\Eab}\\
(x,y)&\mapsto \left( \frac{9(x^3+\frac{4}{3}ax^2-4abx+4ab^2)}{x^2}, \frac{27y(x^3+4abx-8ab^2)}{x^3}\right).
\end{split}
\end{equation}
Set $d=4a+27b$.
The dual isogeny $\widehat{\psi}$ is given by 
\begin{equation}
\label{dual psi-isogeny}
\begin{split}
\widehat{\psi}:\widehat{\Eab}&\longrightarrow\Eab\\
(x, y)&\mapsto\left(\frac{x^3-36ax^2+108adx-108ad^2}{3^4x^2}, \frac{y(y^3-108adx+216ad^2)}{3^6x^3} \right).
\end{split}
\end{equation}

Now we define the $\psi$-Selmer group of $\Eab$.
We have the following exact sequence for any field $F$ induced from the exact sequence $0\rightarrow\Eab[\psi]\rightarrow\Eab\rightarrow \widehat{\Eab}\rightarrow 0$, 
\[
0\longrightarrow \widehat{\Eab}(F)/\psi(\Eab(F))\longrightarrow H^1(G_F, \Eab[\psi])\longrightarrow H^1(G_F, \Eab)[\psi]\longrightarrow 0.
\]
The $\psi$-Selmer group of $\Eab$ over a number field $K$ is defined as follows: 
\[
\Sel^\psi(\Eab/K):=\{ \alpha\in H^1(G_K, \Eab[\psi]) \mid \res_\fq(\alpha)\in \Image(\delta_{\psi, K_\fq}) \text{ for every}\,  \fq \in \Sigma_K \}.
\]

The following theorem is an analogue of Theorem~\ref{structure of kummer image Type 1} for the Type II curves.
\begin{theorem}
\label{structure of kummer image Type 2}
Let $\Eab$ and $\widehat{\Eab}$ be elliptic curves over $K$ and $\psi$ be the $3$-isogeny defined in \eqref{psi-isogeny}.
Let $\fq$ be a prime of $K$ and set $P=(x, y)\in \widehat{\Eab}(K_\fq)$.
Write 
\[
\\delta_{\psi, K_\fq}:\widehat{\Eab}(K_\fq)\longrightarrow  \widehat{\Eab}(K_{\fq})/\phi(\Eab(K_{\fq})) \hookrightarrow H^1(G_{K_{\fq}}, \widehat{\Eab}[\psi]) \simeq (L_{\fq}^\ast/L_{\fq}^{\ast 3})_{\Norm=1}
\]
to denote the Kummer map and set $d=4a+27b$.
\begin{enumerate}[\textup{(}i\textup{)}]
\item If $a\notin K_\fq^{\ast 2}$ then, 
\[
\delta_{\psi, K_\fq}(P)=\begin{cases}
1 & \text{if } P=O \\
y-\p^3\sqrt{a}(x-d) & \text{otherwise}.
\end{cases}
\]
\item If $a\in K_\fq^{\ast 2}$ then, 
\[
\delta_{\psi, K_\fq}(P)=\begin{cases}
1 & \text{if } P=O \\
\left(2\p^3\sqrt{a}d, \frac{1}{2\p^3\sqrt{a}d}\right) & \text{if } P=(0, \p^3\sqrt{a}d) \\
\left(-\frac{1}{2\p^3\sqrt{a}d}, -2\p^3\sqrt{a}d\right) & \text{if } P=(0, -\p^3\sqrt{a}d) \\
\left(y-\p^3\sqrt{a}(x-d), y+\p^3\sqrt{a}(x-d)\right) & \text{otherwise}.
\end{cases}
       \]
   \end{enumerate}
\end{theorem}

\begin{proof} The proof can be found in \cite[Chapters 13 and 14]{Ca91}.
\end{proof}

Since the elliptic curve $\Eab$ is not self-isogenous, to calculate the size of the Kummer image we must compute the Tamagawa number; see Lemma~\ref{size of kummer image}.
\subsection{Computation of Tamagawa numbers and image of the Kummer map}
We run Tate's algorithm \cite{Tat75} to compute the Tamagawa numbers.
First consider the situation when $\fq\nmid 2, 3$.

Consider the elliptic curve $\Eab:y^2=x^3+a(x-b)^2$ with discriminant $\Delta(\Eab)=-2^4a^2b^3d$, where $d=4a+27b$.
Writing the elliptic curve in the long Weierstrass form, we have that
\[
a_1=a_3=0, \ a_2=a, \ a_4=-2ab, \ \text{and} \ a_6=ab^2.
\]
We can also calculate the following constants (see \cite[III.1, p.~42]{Sil09})
\[
b_2=4a, \ b_4=-2^2ab, \ b_6=2^2ab^2, \ \text{and} \ b_8=0.
\]
If $\fq\nmid ab$ and $\fq\nmid d$, then $\fq\nmid \Delta(\Eab)$ and $\Eab$ is in minimal form at $\fq$.
Since $\fq$ is a prime of good reduction $c_\fq(\Eab)=1$.
We now run Tate's Algorithm in the case that $a\in K_\fq^{\ast 2}$ and $\fq\mid a$.
In this case, we note that $v_\fq(4ab^2)= 0, 2$  or $4\pmod{6}$.
In the description that follows we assume that $\fq\nmid 2$; however, we note that a similar algorithm exists even when $\fq\mid 2$.
\begin{enumerate}[Step 1:]
\item Since $\fq\mid a$ it follows that $\fq\mid \Delta(\Eab)$.
Go to the next step.
\item Define $\widetilde{\Eab}= \Eab\pmod{\fq}$.
Since $\fq\mid a$, we have $\widetilde{\Eab}:y^2=x^3$ and $\fq \mid a_3, a_4, a_6,$ and $b_2$.
Go to the next step.
\item \label{step 3}
We check that $\fq\mid b_2$ and $\fq^2\mid a_6$ (because $a\in K_\fq^2$) and go to the next step.
\item We check that $\fq^2\mid a_6$ and $\fq^3\mid b_8$ and go to the next step.
\item Suppose we are now in the situation that $\fq^3\nmid b_6$, i.e., $v_\fq(4ab^2)=2\pmod{6}$; this happens when $\fq\mid a$ and $\fq\nmid b$.
Since $\fq^3\mid b_8$ and $ab^2/\fq^2\in K_\fq^{\ast 2}$, we deduce that the splitting field of the polynomial $X^2-(ab^2/\fq^2)\pmod{\fq}$ is same as $\kappa_{K_\fq}$.

In this case, the elliptic curve $\Eab$ is of Kodaira type IV and $c_\fq(\Eab)=3$.
\item Suppose we are now in the situation that $\fq^3\mid b_6$.
Consider the polynomial 
\[
P(X)=X^3+(a/\fq)X+(-2ab/\fq^2)X+ab^2/\fq^3.
\]
Since $\fq^3\mid b_6$, either $v_\fq(4ab^2)\equiv0$ or $4\pmod6$.
In either case,  $P(X)=X^3 \pmod{\fq}$, i.e. $P(X)$ has a triple root in $\kappa_{K_\fq}$.
\item If $v_\fq(4ab^2)=4$ then splitting field of the polynomial $X^2-(ab^2/\fq^4)$ has two distinct roots in $\kappa_{K_\fq}$.
In this case, the elliptic curve is of Kodaira type IV* and $c_\fq(\Eab)=3$.
\item If $v_\fq(4ab^2)\equiv0\pmod{6}$, the original equation of $\Eab$ is not in the minimal form.
We perform the substitution $x=\fq^2x'$ and $y=\fq^3y'$ define $a_{i,r}=a_i/\fq^r$.
Then, we get the following equation of the elliptic curve:
\[
y'^2=x'^3+a_{2,2}x'^2+a_{4,4}x'+a_{6,6}.
\]
Here, $\fq\nmid a_{6,6}$.
Then, $\Eab$ is an elliptic curve of type II and $c_\fq(\Eab)=1$.
\end{enumerate}

Similarly, we can perform Tate's Algorithm for other cases to compute the Tamagawa number.

The next goal of this section is to compute the image of the Kummer map $\delta_{\psi, K_\fq}$ using Lemma~\ref{size of kummer image} and Proposition~\ref{size of A/A^3}.
We summarize all the relevant situations in the following lemma. 

\begin{lemma}
\label{tamagawa number for q not dividing 3}
Let $\fq\nmid 3$ be a prime in $K$ and $a\in K_\fq^{\ast 2}$.
\begin{enumerate}[\textup{(}i\textup{)}]
\item Suppose further that $\fq\mid a$.
\begin{enumerate}[\textup{(}a\textup{)}]
\item If $v_\fq(4ab^2)=0\pmod{6}$ then,  $c_\fq(\Eab)=c_\fq(\widehat{\Eab})=1$.
Moreover, $\frac{\widehat{\Eab}(K_\fq)}{\psi(\Eab(K_\fq))}\cong\left(\frac{A_\fq^\ast}{A_\fq^{\ast 3}}\right)_{\Norm=1}$.

\item If $v_\fq(4ab^2)=2\text{ or }4\pmod{6}$ then,  $c_\fq(\Eab)=c_\fq(\widehat{\Eab})=3$.
In this case, 
\[
\frac{\widehat{\Eab}(K_\fq)}{\psi(\Eab(K_\fq))}\cap\left(\frac{A_\fq^\ast}{A_\fq^{\ast 3}}\right)_{\Norm=1}=\{1\}.
\]
\end{enumerate}
\item Suppose that $\fq\nmid 2a$.
\begin{enumerate}[\textup{(}a\textup{)}]
\item If $\fq\mid b$ then,  $c_\fq(\Eab)= 3v_\fq(b)\text{ and }c_\fq(\widehat{\Eab})=v_\fq(b)$.
In this case 
\[
\{1\}=\frac{\widehat{\Eab}(K_\fq)}{\psi(\Eab(K_\fq))}\subset\left(\frac{A_\fq^\ast}{A_\fq^{\ast 3}}\right)_{\Norm=1}.
\]
\item If $\fq\mid d$ then,  $c_\fq(\Eab)= v_\fq(d)\text{ and }c_\fq(\widehat{\Eab})=3v_\fq(d)$.
In this case, 
\[
\left(\frac{A_\fq^\ast}{A_\fq^{\ast 3}}\right)_{\Norm=1}\subset\frac{\widehat{\Eab}(K_\fq)}{\psi(\Eab(K_\fq))}.
\]
\end{enumerate}
\item Suppose that $\fq\mid 2$ and $\fq\nmid a$.
\begin{enumerate}[\textup{(}a\textup{)}]
    \item If either $\fq\nmid b$ or $\fq \lVert b$ then, $c_\fq(\Eab)=c_\fq(\widehat{\Eab})=3$.
    In this case,
    \[
    \frac{\widehat{\Eab}(K_\fq)}{\psi(\Eab(K_\fq))}\cap\left(\frac{A_\fq^\ast}{A_\fq^{\ast 3}}\right)_{\Norm=1}=\{1\}.
    \]
    \item If $\fq^2 \mid\mid b$ then $c_\fq(\Eab)=v_\fq(d)-2$ and $c_\fq(\widehat{\Eab})=3(v_\fq(d)-2)$.
    In this case, 
    \[
    \left(\frac{A_\fq^\ast}{A_\fq^{\ast 3}}\right)_{\Norm=1}\subset\frac{\widehat{\Eab}(K_\fq)}{\psi(\Eab(K_\fq))}.
    \]
    \item If $\fq^3\mid b$, then $c_\fq(\Eab)=3(v_\fq(d)-2)$ and $c_\fq(\widehat{\Eab})=v_\fq(d)-2$.
    In this case \[
    \{1\}=\frac{\widehat{\Eab}(K_\fq)}{\psi(\Eab(K_\fq))}\subset\left(\frac{A_\fq^\ast}{A_\fq^{\ast 3}}\right)_{\Norm=1}.\]
\end{enumerate}
\end{enumerate}
\end{lemma}

\begin{proof}
For each part, we used Tate's algorithm described previously to conclude the assertion pertaining to the Tamagawa number.
We also emphasize that we give a detailed proof for only one of the cases, the other calculations can be done in exactly the same way.

Suppose that $\fq\nmid 2, 3$ and  $\fq\mid a$ with $v_\fq(4ab^2) \equiv 0\pmod{6}$, then $v_\fq(2\sqrt{a}b) \equiv 0\pmod{3}$.
In this case, Lemma~\ref{size of kummer image} asserts that $\abs{\widehat{\Eab}(K_\fq)[\psi]}=\abs{\widehat{\Eab(K_\fq)}/\psi(\Eab(K_\fq))}=3$. 

Now, using Theorem~\ref{structure of kummer image Type 2}, we deduce that $v_\fq(\delta_{\psi, K_\fq}(P))\equiv0\pmod{3}$  for all $P\in \widehat{\Eab}(K_\fq)[\psi]$.
A similar calculations as in Theorem~\ref{ bounds for selmer}(i-b),  shows that $\delta_{\psi, K_\fq}(\widehat{\Eab})\subset \left(\frac{A_\fq^\ast}{A_\fq^{\ast 3}}\right)_{\Norm=1}$.
Again by Lemma~\ref{size of kummer image} and Proposition~\ref{size of A/A^3}, we get that $\frac{\widehat{\Eab}(K_\fq)}{\psi(\Eab(K_\fq))}$ and $\left(\frac{A_\fq^\ast}{A_\fq^{\ast 3}}\right)_{\Norm=1}$ have same size.
Hence, $\frac{\widehat{\Eab}(K_\fq)}{\psi(\Eab(K_\fq))}\cong\left(\frac{A_\fq^\ast}{A_\fq^{\ast 3}}\right)_{\Norm=1}$. 
\end{proof}

Next, we calculate the image of the Kummer map $\delta_{\psi, K_\fq}$.

\begin{Proposition}
\label{image kummer map}
Suppose that $\fq\nmid 3$.
\begin{enumerate}[\textup{(}i\textup{)}]
\item If $a\notin K_\fq^{\ast 2}$, then $\delta_{\psi, K_\fq}(\widehat{\Eab}(K_\fq))=(A^\ast_\fq/A_\fq^{\ast 3})_{\Norm=1}$.
\item If $a\in K_\fq^{\ast 2}$ and $\Eab$ has good reduction at $\fq$, then $\delta_{\psi, K_\fq}(\widehat{\Eab}(K_\fq))\cong(A^\ast_\fq/A_\fq^{\ast 3})_{\Norm=1}$.

\end{enumerate}
\end{Proposition}

\begin{proof}
\begin{enumerate}[\textup{(}i\textup{)}]
    \item Since $a\notin K_\fq^{\ast 2}$, so $\Eab(K_\fq)[\psi]=\{\cO\}$.
    From Lemma~\ref{size of kummer image}, we have that 
    \[
    \abs{\frac{\widehat{\Eab}(K_\fq)}{\psi(\Eab(K_\fq))}}=\frac{c_\fq(\widehat{\Eab})}{c_\fq(\Eab)}\text{ and }\abs{\frac{\Eab(K_\fq)}{\widehat{\psi}(\widehat{\Eab}(K_\fq))}}=\frac{c_\fq(\Eab)}{c_\fq(\widehat{\Eab})}.
    \]
    Since both the quantities are integers and reciprocal to each other.
    Therefore image of the Kummer map is trivial.
    Now the result follows from Lemma~\ref{size of A/A^3}.
    \item Since $\Eab$ has good reduction at $\fq$, so $c_\fq(\Eab)=c_\fq(\widehat{\Eab})=1$.
    Since $a\in K_\fq^{\ast 2}$, therefore $\abs{\Eab(K_\fq)[\psi]}=3$  and then by Lemma~\ref{size of kummer image}, we get that $\abs{\delta_{\psi, K_\fq}(\widehat{\Eab}(K_\fq))}=\abs{\frac{\widehat{\Eab(K_\fq)}}{\psi(\Eab(K_\fq))}}=3$.
    Now from Lemma~\ref{structre of kummer im at good prime} and Proposition~\ref{size of A/A^3}, we get that $\delta_{\psi, K_\fq}(\widehat{\Eab}(K_\fq))\cong(A_\fq/A_\fq^{\ast 3})_{\Norm=1}$.
    \qedhere
\end{enumerate}   
\end{proof}

\subsection{Main Results for Type II Elliptic Curves.}
Define the following set of primes of $K$, 
\begin{align*}
S_1^\prime (K)&:=\{\fq\in \Sigma_K \colon a\in K_\fq^{\ast 2}, \fq\mid a \text{ and } 6\nmid v_\fq(4ab^2)\} \cup \{ \fq\in \Sigma_K \colon a\in K_\fq^{\ast 2}, \fq\mid b \text{ and }\fq\nmid 2a  \} \\
& \hspace{0.5cm} \cup \{\fq\mid 2 \colon a\in K_\fq^{\ast 2},  \fq\nmid a \text{ and } \fq^2\nmid b \text{ or } \fq^3\mid b \}\\
S_2^\prime(K)&:=\{\fq\in \Sigma_K \colon a\in K_\fq^{\ast 2}, \fq\mid a \text{ and } 6\nmid v_\fq(4ab^2)\} \cup \{ \fq\in \Sigma_K \colon a\in K_\fq^{\ast 2}, \fq\mid d \text{ and }\fq\nmid 2a  \} \\
& \hspace{0.5cm} \cup \{\fq\mid 2 \colon a\in K_\fq^{\ast 2}, \fq\nmid a \text{ and } \fq^2\nmid b \text{ or } \fq^2\lVert b \}
\end{align*}

Intuitively, we choose the set $S_1'(K)$ such that  $(A_\fq^\ast/A_\fq^{\ast 3})_{\Norm=1}$ is \textit{not contained} in the image of the Kummer map.
Whereas, the choice of $S_2'(K)$ is in a way such that the image of the Kummer map $\delta_{\psi, K_{\fq}}$ is \emph{not contained} in $(A_\fq^\ast/A_\fq^{\ast 3})_{\Norm=1}$.

This next result is an analogue of Theorem~\ref{ bounds for selmer} for Type II curves.

\begin{theorem}
\label{bound for type II E.C.}
Recall from Definition~\ref{defi: the sets M and N} the description of the sets $M(-)$ and $N(-)$.
Further set $S_1(K)=S_1^\prime(K)\cup S_3(K)$ and $S_2(K)=S_2^\prime(K)\cup S_3(K)$.
\begin{enumerate}[\textup{(}i\textup{)}]
\item If $a\notin K^{\ast 2}$, then
\[
M(L, S_1, a)\subset \Sel^\psi(\Eab/K)\subset N(L, S_2, a).
\]
\item If $a\in K^{\ast 2}$, then
\[
M(K, S_1, a)\subset \Sel^{\psi}(\Eab/K)\subset N(K, S_2, a).
\]
  \end{enumerate}
\end{theorem}

\begin{proof}
\begin{enumerate}[\textup{(}i\textup{)}]
    \item Consider the case that $a\notin K^{\ast 2}$.
    First we show that $M(L, S_1, a)\subset \Sel^\psi(\Eab/K)$.
    By the alternate definition of $M(L, S_1, a)$ proved in Proposition~\ref{M(s,a) and N(s, a)}, it is enough to  show that, $(A^\ast_\fq/A_\fq^{\ast 3})_{\Norm=1}\subset\delta_{\psi, K_\fq}(\widehat{\Eab}(K_\fq))$ for all $\fq\notin S_1(K)$. 
    
    If $a\notin K_\fq^{\ast 2}$ then Proposition~\ref{image kummer map}  asserts that the required containment is automatically true for all $\fq\notin S_1(K)$.
    On the other hand, if $a\in K_\fq^{\ast 2}$ we see that $\fq\notin S_1(K)$ precisely when $\fq\nmid 3$ and the following situations occur simultaneously 
    \begin{itemize}
        \item either $\fq\nmid a$ or $6\mid v_\fq(4ab^2)$.
        \item either $\fq\nmid b$ or $\fq\mid 2a$
        \item if $\fq\mid 2$ then either $\fq\mid a$ or $\fq^2\mid b$.
    \end{itemize}
    Equivalently, if $a\in K_\fq^{\ast 2}$ then $\fq\notin S_1(K)$ in the following three cases
    \begin{itemize}
    \item  $\fq\nmid 3$, $\fq\mid a$, and $6\mid v_\fq(4ab^2)$. 
    \item $\fq\nmid 2, 3$ and $\fq\nmid a, b$. Here it may possible that $\fq\mid d $. 
    \item In addition, for $\fq\mid 2$, $\fq\nmid a$ and $\fq^2\mid b$ and $\fq^3\nmid b$.
    This means that $\fq\nmid a$ and $\fq^2\lVert b$.       
    \end{itemize}

    By Lemma~\ref{tamagawa number for q not dividing 3}(i-a),  (ii-b), and (iii-b) (and Proposition~\ref{image kummer map}, for the good case), we conclude that  $(A^\ast_\fq/A_\fq^{\ast 3})_{\Norm=1}\subset\delta_{\psi, K_\fq}(\widehat{\Eab}(K_\fq))$ for all $\fq\notin S_1(K)$.
    Thus, $M(L, S_1, a)\subset \Sel^\psi(\Eab/K)$.
    
    If $\fq\notin S_2(K)$, then using Lemmas~\ref{size of A/A^3} and \ref{tamagawa number for q not dividing 3} and an argument similar to above, we get that $\delta_{\psi, K_\fq}(\widehat{\Eab}(K_\fq))\subset(A^\ast_\fq/A_\fq^{\ast 3})_{\Norm=1}$ and hence $\Sel^\psi(\Eab/K)\subset N(L, S_2, a).$  
    \item  Let $a\in K^{\ast 2}.$ By Lemma~\ref{tamagawa number for q not dividing 3} and Proposition~\ref{image kummer map} we get that 
    \[
    M(K, S_1, a)\subset \Sel^{\psi}(\Eab/K)\subset N(K, S_2, a).
    \qedhere
    \]
\end{enumerate} 
\end{proof}

\begin{theorem}
\label{main thm Type II}
Let $S_1(K)$ and $S_2(K)$ are as above, then the following statements hold.
\begin{enumerate}[\textup{(}i\textup{)}]
\item If $a\notin K^{\ast 2} $, then 
\begin{multline*}
\dim_{\F_3}\Cl_{S_1(K)}(K)[3]\leq \dim_{\F_3}\Cl_{S_1(L)}(L)[3] \leq \dim_{\F_3}\Sel^\psi(\Eab/K) \\
\leq \dim_{\F_3}\Cl_{S_2(L)}(L)[3]+\abs{S_2(L)}+r_1+r_2.
\end{multline*}
\item If $a\in K^{*2}$, then 
\[
\dim_{\F_3}\Cl_{S_1(K)}(K)[3]\leq \dim_{\F_3}\Sel^\psi(\Eab/K)\leq \dim_{\F_3}\Cl_{S_2(K)}(K)[3]+\abs{S_2(K)}+2.
\]
\end{enumerate}
\end{theorem}

\begin{proof}
The proof follows from Theorem~\ref{bound for type II E.C.} and Proposition~\ref{3-ranks of M(S, a) and N(S, a)}.
\end{proof}

Similarly, we compute the bounds for $\Sel^{\widehat{\psi}}(\widehat{\Eab}/K)$.
Recall that, 
\[
\widehat{\Eab}=\EC_{-27a, d}, \quad \textrm{ where } d= 4a+27b.
\]
Note that we may swap $b$ and $d$ in the description of the sets $S_1(K)$ and $S_2(K)$. 
Indeed, this is because $4(-27a)+27d=-4 \cdot 24\cdot a+27(4a+27b)=27b$.
Then, it is easy to see that $S_1(K)$ and $S_2(K)$ are swapped.
The following theorem is proved in the same way as before.

 \begin{theorem}
\label{main thm Type II(A)}
Let $S_1(K)$ and $S_2(K)$ be as defined above.
\begin{enumerate}[\textup{(}i\textup{)}]
\item If $a\notin K^{\ast 2} $, then 
\begin{align*}
\dim_{\F_3}\Cl_{S_2(K)}(K)[3]\leq \dim_{\F_3} &\Cl_{S_2(L)}(L)[3]\leq \\
&\dim_{\F_3}\Sel^{\widehat{\psi}}(\widehat{\Eab}/K)\leq \dim_{\F_3}\Cl_{S_1(L)}(L)[3]+\abs{S_1(L)}+r_1+r_2.
\end{align*}
\item If $a\in K^{*2}$, then 
\[
\dim_{\F_3}\Cl_{S_2(K)}(K)[3]\leq \dim_{\F_3}\Sel^{\widehat{\psi}}(\widehat{\Eab}/K)\leq \dim_{\F_3}\Cl_{S_1(K)}(K)[3]+\abs{S_1(K)}+2.
\]
\end{enumerate}
\end{theorem}

Finally, \cite[Lemma~6.1]{SS03} gives us the following corollary related to $3$-Selmer groups of $\Eab$ over $K$, which we denote by $\Sel^3(\Eab/K)$.

\begin{cor}
\label{cor: 4.8}
With notation as before,
\begin{enumerate}[\textup{(}i\textup{)}]
\item If $a\notin K^{\ast 2} $, then 
\begin{align*}
\max\{\rank(\Eab(K)), &\dim_{\F_3}\Cl_{S_1(L)}(L)[3]\} \leq \dim_{\F_3}\Sel^3(\Eab/K)\leq \\
&\dim_{\F_3}\Cl_{S_1(L)}(L)[3]+\dim_{\F_3}\Cl_{S_2(L)}(L)[3]+\abs{S_1(L)}+\abs{S_2(L)}+2(r_1+r_2). 
\end{align*}
\item If $a\in K^{*2}$, then 
\begin{align*}
\rank(\Eab(K))\leq \dim_{\F_3} &\Sel^3(\Eab/K)\leq \\
&\dim_{\F_3}\Cl_{S_1(K)}(K)[3] + \dim_{\F_3}\Cl_{S_2(K)}(K)[3]+\abs{S_1(K)}+\abs{S_2(K)}+4.
\end{align*}
\end{enumerate}
\end{cor}

\subsection{Example}
Let 
$K = \Q(\mu_3, 51^{1/3})$ and consider the elliptic curve
\[
\EC:  y^2 = x^3 + 16(x-1)^2 = x^3 + 16x^2 -32x + 16.
\]
This is the curve \href{https://www.lmfdb.org/EllipticCurve/Q/91/b/2}{91b1} in the Cremona table of rank 1 over $\Q$.
This is a type II curve with $a=16$, $b=-1$, and $d=4a+27b=37$.
We can check that $a\in K^{*2}$ but $S_1'(K) = \emptyset$.
We can easily check via SAGE that $(3)= \fp^6$ for a principal prime ideal $\fp$.
So, $S_1(K)=\{\fp\}$.

On the other hand $(37) = \fq_1 \fq_2 \ldots \fq_6$, so $S_2(K)=\{\fp, \fq_1, \ldots, \fq_6\}$.
Using SAGE, we compute $\dim_{\F_3} \Cl_{S_1(K)}[3] = 2$ and $\dim_{\F_3} \Cl_{S_2(K)}[3] = 0$.
It follows from Theorem~\ref{main thm Type II(A)} that
\[
0\le\dim_{\F_3}\Sel^{\widehat{\psi}}(\widehat{\Eab}/K)\le 5.
\]
Furthermore, in view of Corollary~\ref{cor: 4.8} we further conclude that
\[
1 \le \rank(\Eab(K))\leq \dim_{\F_3} \Sel^3(\Eab/K)\leq 2+0+1+7+4 = 14.
\]
Here, we were unable to compute the rank of $\EC_{a,b}/K$ or even the rank of $\EC_{a,b}/\Q(51^{1/3})$.
However, using SAGE we compute that $\EC/\Q(\mu_3)$ remains of rank 1.

\section{Applications}
\label{S: applications}
We begin this section by recording a result of Gerth which shows that the 3-part of the class group can become arbitrarily large, as we vary over number fields $K=\Q(\mu_3, m^{1/3})$.
More precisely,

\begin{theorem}[{\cite[Theorem~5.1]{Ger75}}]
\label{Gerth}
Let $m=3^s\underset{i=1}{\prod}^Jq_i^{t_i}$, where $t_i=1 \text{or }2$ for all $1\leq i\leq J$ and $s=0, 1, \text{or }2$ and $q_i\equiv -1\pmod{3}$.
Let $K=\Q(\mu_3, m^{1/3})$.
Then $\dim_{\F_3}\Cl(K)[3]$ is $2r$, where 
\[
r=\begin{cases}
J & \text{if every } q_i\equiv -1\pmod{9} \text{ and } m\not\equiv\pm 1\pmod{9}\\
J-1& \begin{cases}
     \text{if every } q_i\equiv -1\pmod{9} \text{ and } m\equiv\pm 1\pmod{9}\\
     \text{if some } q_i\equiv 2 \text{ or } 5\pmod{9} \text{ and } m\not\equiv\pm 1\pmod{9}
     \end{cases}\\
  J-2 & \text{if some } q_i\equiv 2 \text{ or } 5\pmod{9} \text{ and } m\equiv\pm 1\pmod{9}   
\end{cases}.
       \]
\end{theorem}

We now prove a general result which can be used in our setting for the case that $p=3$.

\begin{lemma}
\label{Lim-Murty lemma}
Let $\cF$ be any number field and $S(\cF)$ be a finite set of finite primes.
Then,
\[
\abs{\dim_{\F_p}\Cl(\cF)[p] - \dim_{\F_p}\Cl_S(\cF)[p]} \leq 3\abs{S(\cF)}.
\]
\end{lemma}

\begin{proof}
Our proof is adapted from \cite[Lemma~5.2]{LM16}.
We provide a brief sketch for the convenience of the reader.
Note that \cite[Lemma~10.3.12]{NSW} asserts that there exists an exact sequence
\[
\Z^{\abs{S(\cF)}}\longrightarrow \Cl(\cF) \xrightarrow{\theta_{\cF}} \Cl_S(\cF) \longrightarrow 0.
\]
It is straightforward to notice from the above exact sequence that $\dim_{\F_p}\ker(\theta_{\cF}) \leq \abs{S(\cF)}$.
Since $\ker(\theta_{\cF})$ is a subgroup of the class group it is finite; we can conclude that $\dim_{\F_p}(\ker(\theta_{\cF})/p) \leq \abs{S(\cF)}$.
It follows from \cite[Lemma~3.2]{LM16} that
\[
\abs{\dim_{\F_p}(\Cl(\cF))[p] - \dim_{\F_p}(\Cl_{S}(\cF))[p]} \leq 3\abs{S(\cF)}.
\qedhere
\]
\end{proof}

We now record an immediate corollary of the above discussion.
\begin{cor}
\label{cor: obv from Gerth + LM16}
Keep the notation introduced in Theorem~\ref{Gerth}
Fix a finite set of finite primes of $K$ and denote it by $S(K)$.
Then $\dim_{\F_3}\Cl_{S}(K)[3]\geq 2r - 3\abs{S(K)}$, where $r$ is defined as in Theorem~\ref{Gerth}.
\end{cor}

The next result shows us that the 3-part of the $\phi$-Selmer group of a Type I elliptic curve can also be made arbitrarily large as we vary over all $K = \Q(\mu_3, m^{1/3})$.

\begin{cor}
\label{cor for Type I curve}
Let $\EC_a$ be a Type I elliptic curve and $\phi$ be a rational 3-isogeny.
Let $m=3^s\underset{i=1}{\prod}^Jq_i^{t_i}$ be as defined in Theorem~\ref{Gerth}.
As before, define $K=\Q(\mu_3, m^{1/3})$, fix $S(K) = S_a(K) \cup \{\fp\}$ and set $r$ as in Theorem~\ref{Gerth}.
Then for all $a$,
\[
2r-3\abs{S(K)}\leq \dim_{\F_3} \Cl_{S}(K)[3] \leq \dim_{\F_3}\Sel^\phi(\EC_a/K).
\]
In particular, as $m$ varies over all possible values, $\dim_{\F_3}\Sel^\phi(\EC_a/K)$ becomes arbitrarily large.
\end{cor}

\begin{proof}
The inequality follows immediately from Theorem~\ref{theorem relating phi-Selmer group to class group} and Corollary~\ref{cor: obv from Gerth + LM16}.

For the last observation, we note that as we vary $m$ the number of distinct divisors of $m$ becomes larger.
More precisely, $J \rightarrow \infty$, hence $r\rightarrow \infty$.

\emph{Claim:} $\abs{S(K)}$ is bounded above by a constant independent of $J$.

\emph{Justification:} By definition, it suffices to show that $\abs{S_a(K)}$ is bounded above by a constant independent of $J$.
A na{\"i}ve upper bound on how large the set $S_a(K)$ can be is given by
\[
\abs{S_a(K)} \leq [K:\Q]\times \left( \#\{\fq\in \Sigma_K \mid \fq\mid 4a\}\right) \leq  6^2 \omega(4a).
\]
Here, we use $\omega(n)$ to denote the number of distinct primes factors of $n$.
The claim follows.

We see that as $J\rightarrow \infty$, the quantity $2r- 3\abs{S(K)} \rightarrow \infty$ as well.
In other words, as $m$ varies over all possible values, the 3-part of the $\phi$-Selmer group becomes arbitrarily large.
\end{proof}

We explained in the proof of Corollary~\ref{a finer bound for Type I} that $\dim_{\F_3}\Sel^\phi(\EC_a/K)\le\dim_{\F_3}\Sel^3(\EC_a/K)$.
By Corollary~\ref{cor for Type I curve} we see that as $m$ varies over all possible values, $\dim_{\F_3}\Sel^3(\EC_a/K)$ becomes arbitrarily large.
It would be interesting to pin down whether the large 3-Selmer over $K$ is accounted for by large Mordell--Weil rank of $\EC_a/K$ or a large (3-part) of the Shafarevich--Tate group.
Over $\Q$, we know that (at present) the record for largest rank of a Type I curve is 17 discovered by N.~Elkies in 2016.
However, it is difficult to compute the rank of $\EC_{a}/K$ as $K$ varies.

\begin{remark}
As an application of Theorem~\ref{main thm Type II} an analogue of Corollary~\ref{cor for Type I curve} can be written for Type II elliptic curves as well. 
\end{remark}

\textbf{There is no data associated with this manuscript.}
\bibliographystyle{amsalpha}
\bibliography{references}
\end{document}